\documentclass[12pt, a4paper, leqno]{amsart}

\usepackage[utf8]{inputenc}
\usepackage{amsmath} 
\usepackage{amsfonts}
\usepackage{amssymb}
\usepackage{txfonts}   
\usepackage{amsfonts}     
\usepackage{graphicx}    
\usepackage[dvipsnames]{xcolor}
\usepackage[colorlinks, allcolors=RedViolet]{hyperref}
 
\newcommand{\R}{\varmathbb{R}}

\newcommand{\Rn}{{\varmathbb{R}^n}}

\newcommand{\Ha}{\mathcal{H}}

\newcommand{\C}{\mathcal{C}}

\newcommand{\ve}{\varepsilon}

\def\diam{\qopname\relax o{diam}}
\def\loc{\qopname\relax o{loc}}

\def\dist{\qopname\relax o{dist}}

\def\diam{\qopname\relax o{diam}}

\def\phi{\varphi}

\let\oldmarginpar\marginpar
\renewcommand\marginpar[1]{\-\oldmarginpar[\raggedleft\footnotesize #1]%
{\raggedright\footnotesize #1}}

\usepackage{amsthm}

\swapnumbers
\theoremstyle{plain}
\newtheorem{theorem}[equation]{Theorem}
\newtheorem{lemma}[equation]{Lemma}
\newtheorem{proposition}[equation]{Proposition}
\newtheorem{corollary}[equation]{Corollary}

\theoremstyle{definition}
\newtheorem{definition}[equation]{Definition}

\newtheorem{example}[equation]{Example}

\theoremstyle{remark}
\newtheorem{remark}[equation]{Remark}

\numberwithin{equation}{section}

\title[On Choquet integrals  and Poincar\'e-Sobolev  inequalities]{On Choquet integrals and  Poincar\'e-Sobolev  inequalities }

\author{Petteri Harjulehto}
\address[Petteri Harjulehto]{Department of Mathematics and Statistics,
FI-00014 University of Helsinki, Finland}
\email{petteri.harjulehto@helsinki.fi}

\author{Ritva Hurri-Syrj\"anen}
\address[Ritva Hurri-Syrj\"anen]{Department of Mathematics and Statistics,
FI-00014 University of Helsinki, Finland}
\email{ritva.hurri-syrjanen@helsinki.fi}

\date{\today}

\begin{document}


\begin{abstract} 
We consider integral inequalities  in the sense of Choquet with respect to  the Hausdorff content $\Ha_\infty^{\delta}$.
In particular, 
if $\Omega$ is a bounded John domain in $\R^n$, $n\geq 2$,  and  $0 <\delta \le n$, we prove that the corresponding
$(\delta p/(\delta -p),p)$-Poincar\'e-Sobolev inequalities hold for all continuously differentiable functions defined on $\Omega$
whenever $\delta /n < p < \delta$.
We prove also that the $(p,p)$-Poincar\'e inequality is valid for all $p>\delta /n$.

 \end{abstract}

\maketitle

%%%%%%%%%%%%%%%%%%%%%%%%%%%%%%%%%%%%%%%%%%%%%%%%%%%%%%%
\section{Introduction}
We are working on Euclidean $n$-space $\Rn$, $n\geq 2$.
 We recall the definition of Choquet integrals over sets $E$ in $\Rn$  with respect to the Hausdorff content $\Ha_\infty^{\delta}$
and consider corresponding integral inequalities.
In particular, we are interested in  the Poincar\'e and Poincar\'e-Sobolev inequalities
in this context.

Our main theorem, Theorem~\ref{thm:main-fractional} gives the following corollary.
\begin{corollary}\label{cor:SP}
Let $\Omega$ be a bounded  $(\alpha, \beta)$-John domain in $\Rn$.
If $0<\delta \le n$ and $p \in (\delta/n, \delta)$, then there exists a constant $c$ depending only on $n$, $\delta$, $p$, and John constants $\alpha$ and $\beta$ such that  
\begin{equation}\label{MainIneq}
\inf_{b \in \R}\Big(\int_\Omega |u(x) -b|^{\frac{\delta p}{\delta -p}} d \Ha^{\delta}_\infty \Big)^{\frac{\delta -p}{\delta p}}
\le c    \Big(\int_\Omega |\nabla u(x)|^p d \Ha^{\delta}_\infty\Big)^{\frac{1}{p}}
\end{equation}
for all $u \in C^1(\Omega)$.
\end{corollary}
We will show that the exponent $\delta p/(\delta -p)$ is the best possible exponent in this setting,  Example \ref{AF}.
Theorem~\ref{thm:main-fractional}  states a version of the Poincar\'e-Sobolev inequality \eqref{maineq}  where the dimension of the Hausdorff content is smaller on the left hand side than on the right hand side
 of the corresponding inequality.

We prove also the Poincar\'e inequality for any 
$p>\delta /n$, 
$\delta\in (0,n]$,  
that is,
there is a constant $c$  which depends only on $\delta$, $n$, $p$, and John constants $\alpha$ and $\beta$ such that the inequality
\begin{equation*}
\inf_{b \in \R} \int_\Omega |u(x)-b|^p \, d \Ha^{\delta}_\infty
\le c  \int_\Omega |\nabla u(x)|^p \, d \Ha^{\delta}_\infty
\end{equation*}
holds for all $u \in C^1(\Omega)$
whenever $\Omega$ is a bounded  $(\alpha, \beta)$-John domain,
 Theorem \ref{thm:Poincare}.

We state and prove the corresponding  Poincar\'e- and Poincar\'e-Sobolev -type inequalities for continuously differentiable functions with compact support defined on open, connected sets in
 Theorem \ref{thm:Poincare_0}.

If $\delta =n$,
 our results recover the earlier well-known results, \cite{Bojarski1988}.
Although there is a wealth literature on Poincar\'e- and Poincar\'e-Sobolev -type inequalities
in various contexts, the authors of the present paper have not been able to find  previous results
where the integrals on the both sides of the corresponding inequalitites are in the sense of Choquet with respect to the Hausdorff content $\Ha_{\infty}^{\delta}$,  $0<\delta <n$.

We point out that
there are  Poincar\'e-type inequalities for $C^{\infty}_0(\Rn)$ functions
when only the left hand side  is the Choquet integral 
with respect to Hausdorff content and the right hand side is the usual Lebesgue integral.
We recall the following estimate which is a special case of  the recent result  \cite[Theorem 1.7]{PS} and   
is called the inequality of  D. R. Adams.
%\begin{theorem}\label{PS}
There exists a constant $c(n)$ such that
\begin{equation}\label{equ:Adams}
\int_{\Rn}
\vert u(x)\vert\,d\Ha_{\infty}^{n-1}\le c(n)\vert\vert\nabla u\vert\vert_{L^1(\Rn)}\,
%\vert\u(x)\vert\,d\Ha_{\infty}^{\delta}\le c(n,\delta)[\nabla ^ku]_{W^{n-\delta-k,1}(\Rn)\,
\end{equation}
for every $u\in C^{\infty}_0(\Rn)$.
%Here $k$ is the integer part of $n-\delta$, with the convention that 
%$[\nabla ^ku]_{W^{0,1}(\Rn):=\vert\vert\nabla ^ku\vert\vert_{L^1(\Rn)}$.
%We refer also to \cite[Theorem 1.6]{PS} where the fractional Poincar\'e-Sobolev  inequalities for
%smooth  functions $u\in C^{\infty}(\bar{\Omega})$ are considered.
%There $\Omega$ is a smooth, bounded domain and the left hand side integration is with respect to any non-negative Borel measure
%$\mu\le \Ha_{\infty}^{n-\alpha}$, $\alpha\in (0,1)$.

%%%%%%%%%%%%%%%%%%%%%%%%%%%%%%%%%%%%%%%%%%%%%%%%%%%%%%%%%%%%%%%%%%%%%%%%%%%%%%%%%%%
\section{Hausdorff content and the Choquet integral}

We recall the definition of Hausdorff content of a set $E$ in $\Rn$,  \cite[2.10.1, p.~169]{Federer}. We refer to  \cite{Ad1998} and \cite[Chapter 3]{Ada15}, too.
An  open ball centered at $x$ with radius $r>0$ is written as $B(x,r)$.

\begin{definition}[Hausdorff content]\label{defn:Haus}
Let $E$ be a set in $\Rn$, $n \ge 2$. Suppose that
$\delta \in (0, n]$.
The Hausdorff content of $E$ is defined by
\begin{equation}
\Ha_\infty^{\delta} (E) := \inf \bigg\{ \sum_{i=1}^\infty r_i^{\delta}: E \subset \bigcup_{i=1}^\infty B(x_i, r_i)\bigg\}\label{HausdorffC}
\end{equation}
where the infimum is taken over all  finite or
countable ball coverings of $E$.
The quantity \eqref{HausdorffC} is called also the $\delta$-Hausdorff content or  $\delta$-Hausdorff capacity
 or the Hausdorff content of $E$ of dimension $\delta$.
\end{definition}

The Hausdorff content has the following properties:
\begin{enumerate}
\item[(H1)] $\Ha_\infty^{\delta}(\emptyset) =0$;
\item[(H2)] if $A \subset B$ then $\Ha_\infty^{\delta}(A) \le \Ha_\infty^{\delta} (B)$;
\item[(H3)]  if $E \subset \Rn$ then 
\[
\Ha_\infty^{\delta}(E) = \inf_{E \subset U \text{ and }U \text{ is open}}\Ha_\infty^{\delta}(U); 
\]
\item[(H4)] if $(K_i)$ is a decreasing sequence of compact sets then 
\[
\Ha_\infty^{\delta}\Big(\bigcap_{i=1}^\infty K_i \Big)
= \lim_{i \to \infty} \Ha_\infty^{\delta}(K_i);
\]
\item[(H5)] if $(A_i)$ is any sequence of sets then
\[
\Ha_\infty^{\delta}\Big(\bigcup_{i=1}^\infty A_i \Big)
\le  \sum_{i=1}^\infty \Ha_\infty^{\delta}(A_i).
\]
\end{enumerate}
The proofs of properties  (H1)--(H5) are straightforward.  
Properties (H1), (H2), (H3), and (H5) yield that $\Ha_\infty^{\delta}$ is an outer capacity in the sense of N. Meyers \cite[p. 257]{Mey70}. 
By properties (H1), (H2) and (H5)  the Hausdorff content is an outer measure.

 We point out that the Hausdorff content 
 $\Ha_\infty^{\delta}$
 does not have the following property:
 if $(E_i)$ is an increasing sequence of  sets then 
\begin{equation}
\Ha_\infty^{\delta}\big(\bigcup_{i=1}^\infty E_i \big)
= \lim_{i \to \infty} \Ha_\infty^{\delta}(E_i),\label{Need_to_Choque_cap}
\end{equation}
we refer  to \cite{Dav56}, and also \cite{Dav70, SioS62}. 
Thus the Hausdorff content
$\Ha_\infty^{\delta}$
 is  not a capacity in the sense of Choquet \cite{Cho53}.

Let us recall the dyadic counterpart of $\Ha_{\infty}^{\delta}$, that is
\begin{equation}
\tilde{\Ha}_\infty^{\delta} (E) := \inf \bigg\{ \sum_{i=1}^\infty \ell (Q_i)^{\delta}: E \subset \bigcup_{i=1}^\infty Q_i\bigg\}\label{HausdorffCD}
\end{equation}
where the infimum is taken over all dyadic cube coverings of $E$.
Here $\ell (Q)$ is the side length of a cube $Q$.
It is known that  $\Ha_\infty^{\delta}(E) $ and $\tilde{\Ha}_\infty^{\delta}(E)$ are comparable to each other  for all sets $E$ in $\Rn$,
that is there are finite positive constants $c_1(n)$ and $c_2(n)$ such that
\begin{equation*}
c_1(n)\Ha_\infty^{\delta}(E)\le\tilde{\Ha}_\infty^{\delta}(E) \le c_2(n)\Ha_\infty^{\delta}(E)\,,
\end{equation*}
we refer to \cite{Adm86} and \cite[Chapter 2, Section 7]{Rog70}.
By \cite[Proposition 2.1 and Proposition 2.2]{YanY08} the dyadic Hausdorff content $\tilde{\Ha}_\infty^{\delta}(E)$ is a capacity in the sense of Choquet only when $n-1\le \delta\le n$.
D.~Yang and W.~Yuan overcame this obstacle by defining a new dyadic Hausdorff content 
$\tilde{\tilde{\Ha}}_\infty^{\delta}(E)$  by requiring in \eqref{HausdorffCD} that
$E$ is a subset of the interior of the set $\bigcup_i Q_j$, \cite[Definition 2.1]{YanY08}.
 Now this new dyadic Hausdorff content    $\tilde{\tilde{\Ha}}_\infty^{\delta}(E)$     is a capacity in the sense of Choquet for all $0<\delta\le n$.
By \cite[Proposition 2.3]{YanY08} this new Hausdorff content $\tilde{\tilde{\Ha}}_\infty^{\delta}$ is comparable to the Hausdorff content $\Ha_\infty^{\delta}$, and constants depend only on $n$.
By \cite[Theorem 2.1 and Proposition 2.4]{YanY08}  $\tilde{\tilde{\Ha}}_\infty^{\delta}$ is a strongly subadditive Choquet capacity  for all $0<\delta \le n$.
 For the strongly subadditivity we refer to \cite{Ad1998}.

Let $0 <\delta\le n$. We  recall the definition of the $\delta$-dimensional  Hausdorff measure for $E \subset \Rn$,

\[
\Ha^\delta (E) := \lim_{\rho \to 0^+}  \inf \bigg\{ \sum_{i=1}^\infty r_i^{\delta}: E \subset \bigcup_{i=1}^\infty B(x_i, r_i) \text{ and } r_i \le \rho \text{ for all } i\bigg\},
\]
where the infimum is taken over 
 all such finite or countable ball coverings of $E$ that the radius of a  ball is at most $\rho$.
Thus  there are finite positive constants $c_1(n)$ and $c_2(n)$ such that $c_1(n)\Ha^n(E) \le |E| \le c_2(n)\Ha^n(E)$ for all Lebesgue measurable 
 sets $E$ in $\Rn$.
For the properties of the Hausdorff measure we refer to \cite[Chapter~2]{EvaG92}  and \cite[pp.~54--58]{Mattila}.

The authors  would like to thank Tuomas Orponen
 for clarifying the relationship between the Hausdorff measure and Hausdorff content.
\begin{proposition}\label{prop:Borel-sets}
There exists a constant $c(n)>0$ such that  for all  $E\subset \Rn$ hold
$\Ha_\infty^n(E) \le \Ha^n(E)\le c(n)\Ha_\infty^n(E)$ .
\end{proposition}

\begin{proof}
By definitions we have $\Ha_\infty^n(E) \le \Ha^n(E)$.

For an open ball we have  $\Ha^n(B(x, r)) \le c(n)r^n$.
Fix  $\ve >0$. Let us take an open ball  $B(x_i, r_i)$ covering of $E$ such that
$\sum_i r_i^n \le \Ha^n_\infty(E) + \ve$.
Since $\Ha^n$ is an outer measure, we have
 have by monotonicity and subadditivity  that  
\[
\Ha^n(E) \le \sum_i \Ha^n(B(x_i, r_i)) \le \sum_i c(n) r_i^n \le  c(n) (\Ha^n_\infty(E) + \ve).
\]
Since this holds for all $\ve>0$,  the claim follows.
\end{proof}

\begin{remark}
We use ball coverings for the definition of the $\delta$-dimensional Hausdorff measure.
In the definitions of different Hausdorff contents we use also  ball coverings and dyadic cube coverings as 
 in \cite{Adm86}, \cite{Ad1998}, \cite{DX}, \cite{OV}, \cite{YanY08}.
If one wishes to take coverings with arbitrary sets we refer to  the following result.
The  proof of  \cite[Theorem 2.5]{EvaG92} gives for all measurable $E\subset \Rn$ that
\[
\inf\bigg\{ \sum_{i=1}^\infty \omega(n) \Big(\frac{\diam C_i}{2}\Big)^n : E \subset \bigcup_{i=1}^\infty C_i \bigg\}
= |E|,
\]
where the infimum of the left-hand side is taken over all covering of $E$, and $\omega(n):=  \frac{\pi^{\frac{\delta}2}}{\Gamma(\frac{\delta}2 +1)}$.
\end{remark}

 We recall the definition of the Choquet integral.
 In the present paper $\Omega$  is always assumed to be a domain
 in $ \Rn$, $n \ge 2$,  that is,  an open, connected set.
For a function $f:\Omega\to [0,\infty]$ the integral in the sense of Choquet with respect to Hausdorff content is defined by
\begin{equation}\label{IntegralDef}
\int_\Omega f(x) \, d \Ha^{\delta}_\infty := \int_0^\infty \Ha^{\delta}_\infty\big(\{x \in \Omega : f(x)>t\}\big) \, dt. 
\end{equation}
Note  that $\Ha^{\delta}_\infty$ is monotone.
%i.e. property (H2) holds.
Hence, for  every  function $f:\Omega \to [0, \infty]$
the corresponding distribution function
$t \mapsto \Ha^{\delta}_\infty\big(\{x \in \Omega : f(x)>t\}\big)$ 
is decreasing with respect to $t$.
By decreasing property  we know that  the  distribution function 
$t \mapsto \Ha^{\delta}_\infty\big(\{x \in \Omega : f(x)>t\}\big)$ is measurable
 with respect to Lebesgue measure. Thus, $\int_0^\infty \Ha^{\delta}_\infty\big(\{x \in \Omega : f(x)>t\}\big) \, dt$ 
is well-defined as a Lebesgue integral. 
The right hand side of \eqref{IntegralDef} can be understood  also as an improper Riemann integral.
Although the Choquet  integral is well-defined for non-measurable functions we study here only measurable functions.  We recall that the Choquet integral is  a nonlinear integral and used in non-additive measure theory.

The Choquet integral with respect to Hausdorff content has the following properties:
\begin{enumerate}
\item[(C1)] $ \displaystyle \int_\Omega a f(x) \, d \Ha^{\delta}_\infty = a \int_\Omega  f(x) \, d \Ha^{\delta}_\infty$ for every $a\ge 0$;
\item[(C2)] $\displaystyle \int_\Omega f(x) \, d \Ha^{\delta}_\infty=0$ if and only if $f(x)=0$  for $\Ha^{\delta}_\infty$-almost every $x\in \Omega$;
\item[(C3)] $\displaystyle \int_\Omega \chi_E(x) \, d \Ha^{\delta}_\infty = \Ha^{\delta}_\infty(\Omega \cap E)$;
\item[(C4)] if $A\subset B$, then $\int_A f(x) \, d \Ha^{\delta}_\infty \le \int_B f(x) \, d \Ha^{\delta}_\infty$;
\item[(C5)] if $0\le f\le g$, then $\displaystyle \int_\Omega f(x) \, d \Ha^{\delta}_\infty \le \int_\Omega g(x) \, d \Ha^{\delta}_\infty$;
\item[(C6)] $\displaystyle \int_\Omega f(x)+g(x) \, d \Ha^{\delta}_\infty \le 2\Big(\int_\Omega f(x) \, d \Ha^{\delta}_\infty + \int_\Omega g(x) \, d \Ha^{\delta}_\infty\Big)$;
\item[(C7)] $\displaystyle \int_\Omega f(x)g(x) \, d \Ha^{\delta}_\infty \le 2\Big(\int_\Omega f(x)^p \, d \Ha^{\delta}_\infty\Big)^{1/p} \Big( \int_\Omega g(x)^q \, d \Ha^{\delta}_\infty\Big)^{1/q}$   when   $p, q>1$
are H\"older conjugates, that is $\frac{1}{p}+\frac{1}{q}=1$.
\end{enumerate}
For the proofs of these properties  we refer to \cite{Ad1998} and \cite[Chapter 4]{Ada15} .
% Note that sublinearity i.e.  property  (C6) with  the constant ``1'' instead of ``2'' does not 
%hold by \cite[Theorem 1]{Ad1998}.

Finally, we note that for a function $f:\Omega\to [0,\infty ]$ 
\[
\int_0^\infty \Ha^{\delta}_\infty\big(\{x \in \Omega : f(x)^p>t\}\big) \, dt = \int_0^\infty p t^{p-1}\Ha^{\delta}_\infty\big(\{x \in \Omega : f(x)>t\}\big) \, dt.
\]
Namely, by changing of  the variables, $t^{1/p} = \lambda$ we obtain
\[
\begin{split}
\int_0^\infty \Ha^{\delta}_\infty\big(\{x \in \Omega : f(x)^p>t\}\big) \, dt &=
\int_0^\infty \Ha^{\delta}_\infty\big(\{x \in \Omega : f(x)>t^{1/p}\}\big) \, dt\\
&= \int_0^\infty p \lambda^{p-1 }\Ha^{\delta}_\infty\big(\{x \in \Omega : f(x)>\lambda\}\big) \, d\lambda.
\end{split}
\]

From now on we study functions with values in $[-\infty, \infty]$, and the Choquet integral is taken  of  the absolute value of the function.
We  need the following lemma.

\begin{lemma}\label{lem:OV-lemma}
 Let $\Omega$ be an open subset of $\Rn$ and  let  $0 < \delta \le n$.
 Then there exist constants $c_1(n)$ and $c_2(n)$ such that
\begin{equation}
\frac{1}{c_1(n)}  \int_\Omega |f(x)|  \, d \Ha^{ n}_\infty \le \int_\Omega |f(x)| \, dx \le c_1(n)  \int_\Omega |f(x)|  \, d \Ha^{ n}_\infty
\label{lemma_a}
\end{equation}
and
\begin{equation}
\int_\Omega |f(x)| \, dx \le  \frac{c_2(n)}{\delta} \Big(\int_\Omega |f(x)|^{\frac{\delta}n} \, d \Ha^{\delta}_\infty \Big)^{\frac{n}{\delta}}
\label{lemma_b}
\end{equation}
for all measurable  functions $f: \Omega \to [-\infty, \infty]$.
\end{lemma}

\begin{proof}
By Cavalier's principle we have
\[
\begin{split}
\int_\Omega |f(x)| \, dx &= \int_0^\infty |\{x: |f(x)|>t\}| \, dt
%\le c(n) \int_0^\infty \Ha^n_\infty\Big(\{x: |f(x)|>t\} \Big) \, dt,
\end{split}
\]
and hence the inequalities \eqref{lemma_a} follows by Proposition~\ref{prop:Borel-sets}.

For the inequality \eqref{lemma_b} we need to show that
\[
\int_\Omega |f(x)|  \, d \Ha^{n}_\infty \le \frac{c(n)}{\delta} \Big(\int_\Omega |f(x)|^{\frac{\delta}n} \, d \Ha^{\delta}_\infty \Big)^{\frac{n}{\delta}}.
\]
Let us estimate the  integrand  on the right hand side. The rest of the proof follows by the proof of  \cite[Lemma 3]{OV}.  
 Since  the mapping $t \mapsto t^{\delta/n}$  is concave on $[0,\infty )$,  we have the inequality
$\big(\sum_{i=1}^m r_i^n  \big)^{\frac{\delta}n}
\le  \sum_{i=1}^m  (r_i^n)^{\frac{\delta}n}$, where $r_i>0$.
Thus $(\Ha^n_{ \infty}(E))^{\frac{1}{n}} \le  (\Ha^\delta_{ \infty}(E))^{\frac{1}{\delta}}$.
We obtain by changing the variables that
\[
\begin{split}
\int_0^\infty \Ha^n_\infty\Big(\{x: |f(x)|>t\} \Big) \, dt
&=   \frac{n}{\delta} \int_0^\infty \Ha^n_\infty\Big(\{x: |f(x)|>t^{n/\delta}\} \Big)  t^{\frac{n}{\delta}-1}\, dt \\
&=  \frac{n}{\delta} \int_0^\infty \Ha^\delta_\infty\Big(\{x: |f(x)|^{\delta/n}>t\} \Big)^{\frac{n}\delta}  t^{\frac{n}{\delta}-1}\, dt. 
\end{split}
\]
 If we write  $h(t) := \Ha^\delta_\infty(\{x: |f(x)|^{\delta/n}>t\}$,  the function  $h$ is decreasing. Hence, we obtain
\[
t h(t) \le \int_{0}^t h(t) \, ds \le \int_{0}^t h(s) \, ds  \le \int_0^\infty h(s) \,ds.
\]
 Thus, combining the estimates gives
\[
\begin{split}
\int_0^\infty \Ha^n_\infty\Big(\{x: |f(x)|>t\} \Big) \, dt
&\le  \frac{n}{\delta} \Big( \int_0^\infty h(s) \,ds \Big)^{\frac{n}{\delta} -1} \int_0^\infty h(t)\, dt\\
&\le  \frac{n}{\delta} \Big(  \int_0^\infty h(s) \,ds \Big)^{\frac{n}{\delta}} \\
&=  \frac{n}{\delta}\Big(\int_\Omega |f(x)|^{\frac{\delta}n} \, d \Ha^{\delta}_\infty \Big)^{\frac{n}{\delta}}.
\qedhere
\end{split}
\]
\end{proof}

Let $\kappa \in [0, n)$.
If  $f \in L^1_{\loc}(\Rn)$, the  centered fractional Hardy-Littlewood maximal function of $f$ is  written as 
\[
M_\kappa f(x) := \sup_{r>0} r^{\kappa-n} \int_{B(x, r)} |f(y)| \, dy.
\]
 The non-fractional  centered maximal function $M_0 f$ is written as $Mf$.
If $f$ is a defined only on $\Omega$ in $\Rn$, then $f$ is  defined  to be zero  on $\Rn\setminus\Omega$ in  the definition of $M_\kappa $.

D. R. Adams in 1986 \cite{Adm86}   and J. Orobitg and J. Verdera  in 1998 \cite{OV} proved boundedness of the maximal operator 
 in the sense of  Choquet with respect to Hausdorff content for $p=1$ and $p>\delta /n$, respectively.
These papers as well as  \cite{Ad1998}  seem to assume that the dyadic Hausdorff content is always a Choquet capacity and they  used this to conclude that   the Choquet integral is sublinear.
 However,  Yang and Yan \cite{YanY08}   showed  that the 
dyadic Hausdorff content is a Choquet capacity if and only if the dimension $\delta$ satisfies $n-1<\delta\le n$. 
The modified dyadic Hausdorff content  $\tilde{\tilde{\Ha}}_\infty^{\delta}$ is   a strongly subadditive Choquet capacity  
for all $0<\delta \le n$, and thus       \cite[Theorem 6.3, p.~75]{Den94}, see also   \cite[54.2]{Cho53}  
and \cite[pp.~248--249]{Ang77},  yield
\[
\int_\Omega \sum_{i=1}^{\infty} f_i(x) d \tilde{\tilde{\Ha}}_\infty^{\delta} \le  \sum_{i=1}^{\infty} \int_\Omega  f_i(x) d \tilde{\tilde{\Ha}}_\infty^{\delta}.
\] 
Since $\tilde{\tilde{\Ha}}_\infty^{\delta}$ is comparable with ${\Ha}_\infty^{\delta}$ we obtain
\[
\int_\Omega \sum_{i=1}^{\infty} f_i(x) d \Ha^{\delta}_\infty \le c(n) \sum_{i=1}^{\infty} \int_\Omega  f_i(x) d \Ha^{\delta}_\infty,
\] 
as it is pointed out in \cite[Remark 2.4]{YanY08}. 
Hence   the following theorem holds.

\begin{theorem}[Adams--Orobitg--Verdera]\label{thm:M-bounded}
Let $\delta \in (0, n)$. Then there exists a constant $c$ depending only on $n$, $\delta$, and $p$ such that
for every $p>\delta/n$ and for every $f \in L^1_{\loc}(\Rn)$ we have
\[  
\int_\Rn (Mf(x))^p \, d \Ha^{\delta}_\infty \le c \int_\Rn |f(x)|^p \, d \Ha^{\delta}_\infty.
\]
\end{theorem}

%Comparing to \cite{OV} here is a lower bound for $\delta$ and hidden assumptions of $f$ have been written up.  
Note that in Theorem
\ref{thm:M-bounded}  the exponent $p$ can be smaller than $1$. We need  also the next result by Adams that covers the previous theorem. It shows that the fractional maximal operator is bounded
when the  Choquet integrals are taken with respect to the $\delta$-dimensional Hausdorff content.
We point out  that the dimension  of the Hausdorff content is smaller on the left hand side in  the following inequality than on the right hand side.

\begin{theorem}[Theorem 7(a) of \cite{Ad1998}]\label{thm:fractional-maximal-function}
Suppose that $\delta \in (0, n]$ and $\kappa \in [0, n)$.  If $p \in (\delta/n, \delta /\kappa)$,
then there exists a constant $c$ depending only on $n$, $\delta$, $\kappa$, and $p$ such that
 for every $f \in L^1_{\loc}(\Rn)$ we have
\[  
\int_\Rn (M_\kappa f(x))^p \, d \Ha^{\delta-\kappa p}_\infty \le c \int_\Rn |f(x)|^p \, d \Ha^{\delta}_\infty.
\]
\end{theorem}

%%%%%%%%%%%%%%%%%%%%%%%%%%%%%%%%%%%%%%%%%%%%%%%%%%%%%%%%%%%%%%%%%%%%%%%%%%%%%%%%%%%
%%%%%%%%%%%%%%%%%%%%%%%%%%%%%%%%%%%%%%%%%%%%%%%%%%%%%%%%%%%%%%%%%%%%%%%%%%%%%%%%%%%

\section{Inequalities for $C^1$-functions}

We recall the definition of John domains. 
The  notion  was introduced by F. John  in \cite{J} where it was called  an inner  radius and outer radius property. 
Later, domains with this property were
named as John domains.

\begin{definition}\label{bounded-john}
Suppose that  $\Omega$ is a bounded domain in $ \Rn$, $n\geq 2$. The domain $\Omega$ is an $(\alpha, \beta)$-John domain if there exist constants $0< \alpha \le \beta<\infty$ and a point $x_0 \in \Omega$ such that each point $x\in \Omega$ can be joined to $x_0$ by a rectifiable curve $\gamma_x:[0,\ell(\gamma_x)] \to \Omega$, parametrized by its arc length, such that $\gamma_x(0) = x$, $\gamma_x(\ell(\gamma_x)) = x_0$, $\ell(\gamma_x)\leq \beta\,,$ and
\[
\dist\big(\gamma_x(t), \partial \Omega \big)
\geq \frac{\alpha}{\beta} t
 \quad \text{for all} \quad t\in[0, \ell(\gamma_x)].
\]
The point $x_0$ is called a John center of $\Omega$.
\end{definition}
Examples of John domains are convex domains and domains with Lipschitz boundary, but also domains with fractal boundaries such as the von Koch snow flake.
Outward spires are not allowed.

We show  that the Poincar\'e inequality in the sense of Choquet with respect to  Hausdorff content  is valid  in John domains.
From now on  we denote the integral average of a function $u$ over a ball $B$ by $u_B$ where the integrals are taken with respect to the Lebesgue measure.

\begin{theorem}\label{thm:Poincare}
Suppose that  $\Omega$ is a bounded  $(\alpha, \beta)$-John domain in $\Rn$.
If  $\delta \in (0, n]$ and $p \in (\delta/n, \infty)$, then there exists a constant $c$ depending only on $n$, $\delta$, $p$, 
and John constants $\alpha$ and $\beta$ such that  
\begin{equation}\label{poincare-eq}
\inf_{b \in \R} \int_\Omega |u(x)-b|^p \, d \Ha^{\delta}_\infty
\le c(n,p,\delta )\beta^p\biggl(\frac{\beta}{\alpha}\biggr)^{2np} \int_\Omega |\nabla u(x)|^p \, d \Ha^{\delta}_\infty
\end{equation}
for all $u \in C^1(\Omega)$.
\end{theorem}

\begin{proof}
Suppose  that $0<\delta<n$.
Let $u \in C^1(\Omega)$. We may assume that the right hand side  of the above inequality
 is finite.
By Lemma~\ref{lem:OV-lemma} we have
\[
\int_\Omega |\nabla u(x)|^{\frac{pn}{\delta}} \, dx \le c(n, \delta) \Big( \int_\Omega |\nabla u(x)|^{p} d \Ha^{\delta}_\infty\Big)^{\frac{n}{\delta}}.
\] 
By the assumption $p> \frac{\delta}{n}$.  Hence we have $\frac{pn}{\delta}>1$. This together with the boundedness of $\Omega$ yields that 
 $|\nabla u| \in L^{1}(\Omega)$. Thus the Riesz potential and the maximal function of $|\nabla u|$ are well-defined.

Since $\Omega$ is a John domain  we obtain by \cite[Theorem]{R}, \cite{Bojarski1988}, \cite{Martio},  and \cite{Hurri}
the pointwise estimate
\begin{equation}\label{pointwise}
|u(x) -u_B| \le c(n, \alpha, \beta) \int_\Omega \frac{|\nabla u(y)|}{|x-y|^{n-1}}\,dy
\end{equation}
for every $x \in \Omega$.   Here, $B=B(x_0, c(n)\alpha^2/\beta)$ and
$c(n,\alpha ,\beta)=c(n)(\beta /\alpha)^{2n}$ by \cite{Hurri}.
The Riesz potential can be estimated by the Hardy-Littlewood maximal operator. Thus by \cite[Lemma 2.8.3]{Zie89} we have
\[
|u(x) -u_B| \le c(n,p)\biggl(\frac{\ \beta}{\alpha}\biggr)^{2n} \diam(\Omega) M |\nabla u| (x) 
\]
for every $x \in \Omega$. Hence,  by properties (C5) and (C1) of the Choquet integral we obtain
\[
\int_\Omega |u(x) -u_B|^p d \Ha^{\delta}_\infty
\le  c(n, p)\biggl(\frac{\beta}{\alpha}\biggr)^{2np} \diam(\Omega)^{p} \int_\Omega (M |\nabla u|(x) )^p d \Ha^{\delta}_\infty.
\]
Since  the maximal operator is bounded  in the sense of Choquet with respect to Hausdorff content by Theorem~\ref{thm:M-bounded}, we obtain 
\[
\begin{split}
\int_\Omega |u(x) -u_B|^p d \Ha^{\delta}_\infty
& \le c(n, \delta, p)\biggl(\frac{\beta}{\alpha}\biggr)^{2np} \diam(\Omega)^{p} \int_\Omega |\nabla u(x)|^p d \Ha^{\delta}_\infty. 
\end{split}
\]
Since $\Omega$ is a  bounded John domain, we have $\diam(\Omega)\le 2 \beta$.

If $\delta=n$, then by   \cite{Bojarski1988}, \cite{Martio}, \cite{Hurri} we have
\[
\inf_{b \in \R}\int_\Omega |u(x) -b|^p dx 
\le c(n,p)\beta^p\biggl(\frac{\beta}{\alpha}\biggr)^{2np}   \int_\Omega |\nabla u(x)|^p \, dx
\]
for all $u \in C^1(\Omega)$ with $|\nabla u| \in L^p(\Omega)$. 
Now the claim follows by Lemma~\ref{lem:OV-lemma}.
\end{proof}

\begin{remark}
If $\Omega$ is a  bounded convex domain then the same proof yields  
\[
\inf_{b \in \R} \int_\Omega |u(x)-b|^p \, d \Ha^{\delta}_\infty
\le c(n, \delta, p) \frac{\diam(\Omega)^{np+p}}{|\Omega|^p}\int_\Omega |\nabla u(x)|^p \, d \Ha^{\delta}_\infty
\]
for all $u \in C^1(\Omega)$. Here we have used  \cite[Lemma 7.16]{Gilbarg-Trudinger} instead of the previous estimate for functions in a John domain.
\end{remark}

Next we estimate the Riesz potential by the Hedberg-type pointwise estimate where we have the Choquet integral. The classical version of this pointwise estimate goes back to \cite{Hed72}. We use the fractional maximal function  by following the idea of \cite{Adm75}.

\begin{lemma}\label{lem:Choquet-Hedberg}
Let $\kappa \in [0, 1)$, $\delta \in (0, n]$,  and $p \in (\delta/n, \delta)$. Then there exists a constant $c$ depending only on $n$, $\delta$, $\kappa$, and $p$ such that
\[
\int_\Rn \frac{|f(y)|}{|x-y|^{n-1}} \, dy
\le c 
\bigg(M_\kappa f(x)\bigg)^{\frac{\delta-p}{\delta-\kappa p}}
\bigg(\int_{\Rn} |f(y)|^{p} \, d \Ha^{\delta}_\infty \bigg)^{\frac{1-\kappa}{\delta -\kappa p}}  
\]
for all $x \in \Rn$  and all $f \in L^1_{\loc}(\Rn)$.

\end{lemma}

\begin{proof}
Let $A_k =\{y \in \Rn : 2^{-k} r \le |x-y|< 2^{-k+1}r\}$. We estimate
\[
\begin{split}
\int_{B(x, r)} \frac{|f(y)|}{|x-y|^{n-1}} \, dy 
&= \sum_{k=1}^\infty \int_{A_k} \frac{|f(y)|}{|x-y|^{n-1}} \, dy\\
&\le  \sum_{k=1}^\infty (2^{-k}r)^{1-n} \int_{B(x, 2^{-k+1}r)} |f(y)| \, dy\\
&\le   \frac{2^{\kappa -1 +n}}{1-2^{\kappa -1}} r^{1-\kappa} M_\kappa f(x), 
\end{split}
\]
where in the last step  the sum of a geometric series is used.

Outside the ball  $B(x,r)$ we use Hölder's inequality and Lemma~\ref{lem:OV-lemma} to obtain 
\[
\begin{split}
&\int_{\Rn\setminus B(x, r)} \frac{|f(y)|}{|x-y|^{n-1}} \, dy
\le \Big(\int_{\Rn\setminus B(x, r)} |f(y)|^{\frac{np}{\delta}}\, dy \Big)^{\frac{\delta}{np}} 
\Big(\int_{\Rn\setminus B(x, r)} |x-y|^{\frac{np(1-n)}{np-\delta}} \, dy\Big)^{\frac{np-\delta}{np}}\\
&\le  c(n, \delta, p)\Big(\int_{\Rn\setminus B(x, r)} |f(y)|^{p} \, d \Ha^{\delta}_\infty \Big)^{\frac{1}{p}} 
\Big(\int_{\Rn\setminus B(x, r)} |x-y|^{\frac{np(1-n)}{np-\delta}} \, dy\Big)^{\frac{np-\delta}{np}}.
\end{split}
\]
The last term on the right hand side  is
\[
\int_{\Rn\setminus B(x, r)} |x-y|^{\frac{np(1-n)}{np-\delta}} \, dy
= \frac{\omega_{n-1}}{(n-1)\frac{np(n-1)}{np-\delta} -n } r^{n-\frac{np(n-1)}{np-\delta}},
\]
where $\omega_{n-1}$ is the $n-1$-dimensional Hausdorff measure of the sphere, \cite[Lemma]{Hed72}.
Note that $n-\frac{np(n-1)}{np-\delta}<0$,  since  $p \in (\delta/n, \delta)$.
Thus we have
\[
\int_\Rn \frac{|f(y)|}{|x-y|^{n-1}} \, dy \le c  \big( r^{1-\kappa} M_\kappa f(x) +  \|f\| r^{1-\frac{\delta}{p}} \big),
\]
where $\|f\| := \Big(\int_{\Rn\setminus B(x, r)} |f(y)|^{p} \, d \Ha^{\delta}_\infty \Big)^{\frac{1}{p}}$.
By choosing 
\[
r= \Big( \frac{M_\kappa f(x)}{\|f\|} \Big)^{-\frac{p}{\delta - \kappa p}}
\]
we obtain
\[
\int_\Rn \frac{|f(y)|}{|x-y|^{n-1}} \, dy \le c (M_\kappa f(x))^{1-\frac{p(1-\kappa)}{\delta- \kappa p}} \|f\|^{\frac{p(1-\kappa)}{\delta-\kappa p}}
\]
 for all $x \in \Rn$.
 This inequality yields the claim.
\end{proof}

The previous lemma gives our main result. Note that if $\kappa >0$ then the dimension of the Hausdorff content is lower on the left and side than on the right hand side.

\begin{theorem}\label{thm:main-fractional}
Let $\Omega$ be a bounded  $(\alpha, \beta)$-John domain in $\Rn$.
Suppose that $\delta \in(0, n]$, $\kappa \in [0, 1)$, and $p \in (\delta/n, \delta)$. Then there exists a constant $c$ depending only on $n$, $\delta$, $\kappa$, $p$, 
 and John constants $\alpha$ and $\beta$ such that  
\begin{equation}\label{maineq}
 \inf_{b \in \R}  \Big(\int_\Omega |u(x) -b|^{\frac{p(\delta- \kappa p) }{\delta-p}} d \Ha^{\delta -\kappa p}_\infty \Big)^{\frac{\delta-p}{p(\delta- \kappa p)}}
\le c \Big(\int_{\Omega} |\nabla u(x)|^{p} \, d \Ha^{\delta}_\infty \Big)^{\frac{1}{p}}
\end{equation}
for all $u \in C^1(\Omega)$.
\end{theorem}

\begin{proof}
Suppose  that $0<\delta<n$ and
$u \in C^1(\Omega)$. We may assume that the right hand side   of inequality \eqref{maineq} is finite.
As in the proof of Theorem~\ref{thm:Poincare} we have $|\nabla u| \in L^1(\Omega)$.
By  \cite[Theorem]{R}, \cite{Bojarski1988}, \cite{Martio},  and \cite{Hurri}  for an $(\alpha ,\beta)$-John domain the pointwise estimate
\[
|u(x) -u_B| \le c(n, \alpha, \beta) \int_\Omega \frac{|\nabla u(y)|}{|x-y|^{n-1}} \, dy
\]
holds 
for every $x \in \Omega$.
 Here $B=B(x_0,c(n)\alpha^2/\beta)$ and $c(n,\alpha ,\beta )=c(n)(\alpha /\beta)^{2n}$ by \cite{Hurri}.
Next we apply Lemma~\ref{lem:Choquet-Hedberg} with the understanding that $|\nabla u|$ is  zero outside $\Omega$. We obtain
\[
\begin{split}
|u(x) -u_B|^{\frac{p(\delta- \kappa p) }{\delta-p}} &\le c \bigg( \int_\Omega \frac{|\nabla u(y)|}{|x-y|^{n-1}} \, dy \bigg)^{\frac{p(\delta- \kappa p) }{\delta-p}}\\
&\le c \Big(\int_{\Omega} |\nabla u(y)|^{p} \, d \Ha^{\delta}_\infty \Big)^{\frac{p(1-\kappa)}{\delta -p}}  (M_\kappa |\nabla u| (x))^p 
\end{split}
\]
for every $x \in \Omega$. Here the constants depends on $n$, $\delta$, $\kappa$, $p$,  $\alpha$, and  $\beta$. 
By integrating with respect to $\Ha^{\delta-\kappa p}_\infty$ and using the properties (C5) and (C1) of the Choquet integral we obtain
\[
\int_\Omega |u(x) -u_B|^{\frac{p(\delta- \kappa p) }{\delta-p}} d \Ha^{\delta-\kappa p}_\infty
\le  c \Big(\int_{\Omega} |\nabla u(y)|^{p} \, d \Ha^{\delta}_\infty \Big)^{\frac{p(1-\kappa)}{\delta -p}} \int_\Omega (M_\kappa |\nabla u| )^p d \Ha^{\delta-\kappa p}_\infty.
\]
Adams's result for boundedness of the fractional Hardy-Littlewood maximal operator,
Theorem~\ref{thm:fractional-maximal-function} implies
\[
\begin{split}
\int_\Omega |u(x) -u_B|^{\frac{p(\delta- \kappa p) }{\delta-p}} d \Ha^{d}_\infty
& \le c \Big(\int_{\Omega} |\nabla u(y)|^{p} \, d \Ha^{\delta}_\infty \Big)^{\frac{p(1-\kappa)}{\delta -p}} \int_\Omega |\nabla u(x)|^p d \Ha^{\delta}_\infty\\
&=  c \Big(\int_{\Omega} |\nabla u(x)|^{p} \, d \Ha^{\delta}_\infty \Big)^{\frac{\delta- p \kappa}{\delta -p}}. 
\end{split}
\]
Hence the claim follows by raising both sides of the previous inequality  to the power $\frac{\delta-p}{p(\delta- \kappa p) }$.
\end{proof}

The $(\delta p/(\delta -p), p)$-Poincar\'e-Sobolev inequality 
\begin{equation*}
\inf_{b \in \R}\Big(\int_\Omega |u(x) -b|^{\frac{\delta p}{\delta -p}} d \Ha^{\delta}_\infty \Big)^{\frac{\delta -p}{\delta p}}
\le c    \Big(\int_\Omega |\nabla u(x)|^p d \Ha^{\delta}_\infty\Big)^{\frac{1}{p}}
\end{equation*}
in Corollary~\ref{cor:SP} 
follows  now from Theorem~\ref{thm:main-fractional} when we choose $\kappa =0$.
When $\delta =n$ we recover the classical Sobolev inequality.

\bigskip

Choosing $\delta =n$ and $\kappa =1/p$ in Theorem \ref{thm:main-fractional} gives the following corollary.

\begin{corollary}\label{Corollary_n-1}
Let $\Omega$ be a bounded  $(\alpha, \beta)$-John domain in $\R^n$.
Suppose that  $p \in (1, n)$. Then there exists a constant $c$ depending only on $n$, $p$, 
 and John constants $\alpha$ and $\beta$ such that  
\begin{equation}\label{eq_n-1}
 \inf_{b \in \R}  \Big(\int_\Omega |u(x) -b|^{\frac{p(n- 1) }{n-p}} d \Ha^{n-1}_\infty \Big)^{\frac{n-p}{p(n- 1)}}
\le c \Big(\int_{\Omega} |\nabla u(x)|^{p} \, dx\Big)^{\frac{1}{p}}
\end{equation}
for all $u \in C^1(\Omega)$.
\end{corollary}

\begin{remark}
We point out that  the proofs of 
Theorem~\ref{thm:Poincare} and Theorem \ref{thm:main-fractional}  
give stronger inequalities than \eqref{poincare-eq} and
\eqref{maineq}, respectively.
If  $\Omega$ is a bounded  $(\alpha, \beta)$-John domain in $\Rn$,
$\delta \in(0, n]$, $\kappa \in [0, 1)$, and $p \in (\delta/n, \delta)$, then there exist constants $c_1=c_1(\alpha, \beta, \delta, n, p)$ 
and $c_2=c_2(\alpha, \beta, \delta , \kappa, n, p)$ 
such that the inequalities
\[
\int_\Omega |u(x)-u_B|^p \, d \Ha^{\delta}_\infty
\le c(n,p,\delta )\beta^p\biggl(\frac{\beta}{\alpha}\biggr)^{2np} \int_\Omega |\nabla u(x)|^p \, d \Ha^{\delta}_\infty
\]
and
\[
\Big(\int_\Omega |u(x) -u_B|^{\frac{p(\delta- \kappa p) }{\delta-p}} d \Ha^{\delta -\kappa p}_\infty \Big)^{\frac{\delta-p}{p(\delta- \kappa p)}}
\le c \Big(\int_{\Omega} |\nabla u(x)|^{p} \, d \Ha^{\delta}_\infty \Big)^{\frac{1}{p}}
\]
are valid for all $u\in C^1(\Omega)$.
We recall that $B=B(x_0,c(n)\alpha^2/\beta)$ and the integral average has been calculated  with respect to  the Lebesgue measure. 
\end{remark}

\begin{remark}
Let $\Omega$ be a bounded  $(\alpha, \beta)$-John domain in $\R^n$.
Suppose that  $p \in (1, n)$. 
Choosing $\vert f(x)\vert=\vert u(x)-u_{B}\vert^{\frac{np}{n-p}}$
 in  Lemma \ref{lem:OV-lemma}
gives that there exists a constant $c_1$ such that
\begin{equation*}
\Big(\int_\Omega |u(x) -u_{B}|^{\frac{np }{n-p}} dx \Big)^{\frac{n-p}{pn}}\\
\le c_1
\Big(\int_\Omega |u(x) -u_{B}|^{\frac{p(n- 1) }{n-p}} d \Ha^{n-1}_\infty \Big)^{\frac{n-p}{p(n- 1)}}\,.
\end{equation*}
Corollary \ref{Corollary_n-1} gives that there exists a constant $c_2$ such that we have 
\[
\begin{split}
&\Big(\int_\Omega |u(x) -u_{B}|^{\frac{np }{n-p}} dx \Big)^{\frac{n-p}{pn}}\\
&\le c_1
\Big(\int_\Omega |u(x) -u_{B}|^{\frac{p(n- 1) }{n-p}} d \Ha^{n-1}_\infty \Big)^{\frac{n-p}{p(n- 1)}}\\
&\le c_2
\Big(\int_{\Omega} |\nabla u(x)|^{p} \, dx\Big)^{\frac{1}{p}}\,.
\end{split}
\]
This shows some of the benefits which come from using Choquet integrals in Poincar\'e-Sobolev inequalities.
\end{remark}

\begin{remark}
Note that
by Lemma \ref{lem:Choquet-Hedberg} and Theorem \ref{thm:fractional-maximal-function}
 the Riesz potential $I_1f(x) := \int_\Rn \frac{f(y)}{|x-y|^{n-1}} \, dy$ is bounded  with respect to  Hausdorff content.
If  $0<\delta\le n$, $\kappa \in [0, 1)$, and $p \in (\delta/n, \delta)$, then
\[
\Big( \int_\Rn ( I_1(f(x)))^{\frac{p(\delta- \kappa p)}{\delta-p}}d \Ha^{\delta- \kappa p}_\infty
\Big)^{\frac{\delta-p}{p(\delta- \kappa p)}}
\le c(n, \delta, \kappa)  \Big(\int_{\Rn} |f(x)|^{p} \, d \Ha^{\delta}_\infty \Big)^{\frac{1}{p}}
\]
for all $f \in L^1_{\loc}(\Rn)$.
\end{remark}

Next we show that the exponent $\frac{p(\delta-\kappa p) }{\delta -p}$ in Theorem ~\ref{thm:main-fractional} is the best possible  exponent in this setting. This example is based on the example,  \cite[Example 4.41, p.~109]{AdaF03}.

\begin{example}\label{AF}
Let $\Omega:= B^n(0,1)\setminus\{0\}$,
$0<\delta\le n$, $\kappa \in [0, 1)$, and  $p \in (\delta/n, \delta)$.  Let us define $v(x) := |x|^\mu$, where $\mu<0$ is chosen later. 
Then $v \in C^{\infty}(\Omega)$. 
We show that,
if  $q>\frac{p(\delta- \kappa p) }{\delta -p}$, then there exists
$\mu$  such that  $\int_\Omega |v(x)-a|^{q} d \Ha^{\delta- \kappa p}_\infty = \infty$ for any $a \in \R$ and
 at the same time $\int_\Omega |\nabla v(x)|^p d \Ha^{\delta}_\infty<\infty$.

Let $a \in \R$. For the function $v$ itself we use Lemma~\ref{lem:OV-lemma} to obtain
\[
\begin{split}
c(n, \delta, \kappa, p) \Big( \int_\Omega |v(x)-a|^{q} \, d \Ha^{\delta- \kappa p}_\infty \Big)^{\frac{n}{\delta- \kappa p}}
&\ge \int_\Omega |v(x)-a|^{\frac{qn}{\delta- \kappa p}} \, dx \\
&\ge \int_{B(0, r)} |\tfrac12 v(x)|^{\frac{qn}{\delta- \kappa p}} \, dx\\
&=c(n, \kappa, p, \delta, q) \int_0^r \rho^{\frac{\mu qn}{\delta- \kappa p} + n-1} \, d \rho,
\end{split}
\]
for some $r>0$.
The last integral is infinite whenever $\frac{\mu qn}{\delta- \kappa p} + n-1 \le -1$, that is  if 
$\mu\le  - \frac{\delta- \kappa p}{q}$.  

For the gradient we obtain $|\nabla v(x)| = |\mu| |x|^{\mu -1}$. Thus, $|\nabla v(x)|^p >t$ provided that $|x|< c t^{\frac{1}{p(\mu -1)}}$.
 By using the inequality  $\Ha^{\delta}_\infty(B(0, r)) \le r^\delta$, we obtain
\[
\begin{split}
\int_\Omega |\nabla v(x)|^p \, d \Ha^{\delta}_\infty &= \int_0^\infty \Ha^{\delta}_\infty(\{|\nabla v(x)|^p>t\}) \, dt\\
&\le \Ha^{\delta}_\infty(B(0, 1)) + \int_1^\infty \Ha^{\delta}_\infty(B(0, c t^{\frac{1}{p(\mu -1)}}) \, dt\\
&\le \Ha^{\delta}_\infty(B(0, 1)) + c \int_1^\infty  t^{\frac{\delta}{p(\mu -1)}}\, dt.
\end{split}
\]
The last integral is finite provided that $\frac{\delta}{p(\mu -1)}<-1$ i.e. if $\mu> 1- \frac{\delta}{p}$. Since $q>\frac{p(\delta-\kappa p) }{\delta -p}$, we have
$1- \frac{\delta}{p} < - \frac{\delta- \kappa p}{q}$. Thus we may choose  the parameter $\mu$ such that  $1- \frac{\delta}{p} < \mu \le- \frac{\delta- \kappa p}{q}$.

\end{example}

\begin{remark}
If $\Omega$ is an unbounded domain such that $\Omega =\cup_{i=0}^{\infty}\Omega_i$ where
$\Omega_i\subset \Omega_{i+1}$ and $\Omega_i$ is an $(\alpha_i,\beta_i)$-John domain
 for some $0<\alpha_i\le \beta_i < \infty$, $i=0, 1, \dots\,.$
If $\beta_i/\alpha_i\le c$ for all $i$, then the $(np/n-p),p)$-Poincar\'e-Sobolev inequality holds for all
functions $u\in L^1_p(\Omega)$, \cite[Theorem 4.1]{H-S}. This result corresponds to the case $\delta =n$.
\end{remark}

%%%%%%%%%%%%%%%%%%%%%%%%%%%%%%%%%%%%%%%%%%%%%%%
%%%%%%%%%%%%%%%%%%%%%%%%%%%%%%%%%%%%%%%%%%%%%%%
%%%%%%%%%%%%%%%%%%%%%%%%%%%%%%%%%%%%%%%%%%%%%%%

\section{Inequalities for $C^1_0$-functions}

The Poincar\'e inequality and Poincar\'e-Sobolev inequality for $C^1_0$-functions follow in a similar  fashion as for $C^1$-functions,
respectively.
 The  main difference is to use for functions  $u \in C^1_0(\Omega)$ the estimate
\begin{equation}\label{pointwise2}
|u(x)| \le c(n) \int_\Omega \frac{|\nabla u(y)|}{|x-y|^{n-1}}\,dy \quad\text{ for all } x\in \Rn,
\end{equation}
  \cite[1.1.10 Theorem 2]{Maz85},
instead of  the corresponding  inequality \eqref{pointwise}  for $C^1(\Omega )$-functions  defined on  a John domain.  Inequality  \eqref{pointwise2}  yields the following theorem, where 
the part (a) holds only in a bounded domain while  the part (b) can also be  applied for unbounded domains.
 In fact, if the domain is bounded in the part (b), then  H\"older's  inequality implies the part  (a) too.

\begin{theorem}\label{thm:Poincare_0}
Let $\delta  \in (0, n]$. 
\begin{itemize} 
\item[(a)]
If  $\Omega $ is  a bounded domain  in $\Rn$ 
and  $p \in (\delta/n, \infty)$, then there exists a constant $c$ depending only on $n$, $\delta$, and $p$
 such that  
\[
 \int_\Omega |u(x)|^p \, d \Ha^{\delta}_\infty
\le c \diam(\Omega)^{p} \int_\Omega |\nabla u(x)|^p \, d \Ha^{\delta}_\infty
\]
for all $u \in C^1_0(\Omega)$.
\item[(b)]
If  $\Omega$ is  a domain in $\Rn$, $\kappa \in [0, 1)$, and $p \in (\delta/n, \delta)$, then there exists a constant $c$ depending only on $n$, $\delta$,
$\kappa$, and $p$ such that  
\[
\Big(\int_\Omega |u(x)|^{\frac{p(\delta- \kappa p) }{\delta-p}} d \Ha^{\delta -\kappa p}_\infty \Big)^{\frac{\delta-p}{p(\delta- \kappa p)}}
\le c \Big(\int_{\Omega} |\nabla u(x)|^{p} \, d \Ha^{\delta}_\infty \Big)^{\frac{1}{p}}
\] 
for all $u \in C^1_0(\Omega)$.
\end{itemize}
\end{theorem}

\begin{remark}
Let $\kappa =0$ and $\delta = n-1$.
 Both limit cases $p=\frac{\delta}{n}$ and $p=\delta$ are excluded from Theorem~\ref{thm:Poincare_0}(b).
\begin{itemize}
\item However, by combining  the inequality of Adams \eqref{equ:Adams} and Lemma~\ref{lem:OV-lemma} we obtain the inequality
\begin{equation*}%\label{equ:Adams-2}
\int_{\Rn}
\vert u(x)\vert\,d\Ha_{\infty}^{n-1}\le c(n, \delta) \Big(\int_\Rn |\nabla u(x)|^{\frac{\delta}n} \, d \Ha^{\delta}_\infty \Big)^{\frac{n}{\delta}}
\end{equation*}
for every $\delta \in (0, n]$ whenever $u\in \C^{\infty}_0(\Rn)$.
Note  that if  $p=\frac{\delta}{n}$ and  $\delta =n-1$, then $\frac{\delta p}{\delta -p}=1$.
Hence, the above inequality can be seen as a limit case  if   $p=\frac{\delta}{n}$ with $\delta= n-1$  for Theorem~\ref{thm:Poincare_0}(b)  where $\kappa =0$.

\item 
Corresponding to the upper limiting case $p=\delta=n-1$,  the authors  of the present paper  showed  in  \cite[Corollary 1.3]{HH-S}:
% that $C^1_0$-functions are  exponentially integrable.
If $\Omega$ is a  bounded $(\alpha, \beta )$-John domain in $\Rn$, then
there exist positive constants $a$ and $b$ such that
\begin{equation*}
\int_{\Omega}\exp \big(a\vert u(x)-u_B\vert^{\frac{n}{n-1}}\big)\,d\Ha_{\infty}^{n-1}\le b
\end{equation*}
 for all $u\in L^1_{n}(\Omega) \cap C^1(\Omega)$ with $\vert\vert\nabla u\vert\vert_{L^{n}(\Omega)}\le 1$. Here $B=B(x_0, c(n)\alpha^2/\beta)$.
Moreover,  
$\vert\vert\nabla u\vert\vert_{L^{n}(\Omega)} \le c(n) \big( \int_\Omega |\nabla u|^{n-1} \,d\Ha_{\infty}^{n-1} \big)^{1/(n-1)}$ by Lemma~\ref{lem:OV-lemma}.
\end{itemize}
\end{remark}

%%%%%%%%%%%%%%%%%%%%%%%%%%%%%%%%%%%%%%%%%%%%%%%%%%%%%%
%%%%%%%%%%%%%%%%%%%%%%%%%%%%%%%%%%%%%%%%%%%%%%%%%%%%%%%%%%%%%%%%%%%%

\bibliographystyle{amsalpha}

\end{document}